\documentclass{amsart}
\usepackage{latexsym}
\usepackage{amssymb}
\usepackage{amsfonts}
\usepackage{amsmath}
\usepackage{amsthm}
\usepackage{graphics}
\usepackage[pdftex]{graphicx}
\usepackage{xcolor,soul}
\usepackage{float}
\sethlcolor{green}

\newtheorem{lemma}{Lemma}
\newtheorem{theorem}[lemma]{Theorem}

\newtheorem{definition}{Definition}

\def\gcd{{\rm gcd}}
\def\SL{\rm SL}
\def\PSL{\rm PSL}
\def\Sz{\rm Sz}
\def\cs{\rm cs}
\def\D{\rm D}
\def\K{\rm K}

\begin{document}
\openup 1.8\jot

\title[On divisibility graph for simple Zassenhaus groups]
{On divisibility graph for simple Zassenhaus groups}
\author[A.~Abdolghafourian]{A.~Abdolghafourian}
\address{A.~Abdolghafourian, Department of Mathematics \newline Yazd University
\\ Yazd, 89195-741, Iran}
\email{Adeleh Abdolghafourian@stu.yazd.ac.ir}
\author[Mohammad~A.~Iranmanesh]{Mohammad~A.~Iranmanesh}
\thanks{Corresponding author: Mohammad~A.~Iranmanesh}\address{M.~A.~Iranmanesh, Department of Mathematics \newline Yazd University\\ Yazd, 89195-741 , Iran}
\email{iranmanesh@yazd.ac.ir}

\begin{abstract}
The divisibility graph $\D(G)$ for a finite group $G$ is a graph with vertex set $\cs(G)\setminus\{1\}$
where $\cs(G)$ is the set of conjugacy class sizes of $G$. Two vertices $a$ and $b$ are adjacent whenever $a$ divides
$b$ or $b$ divides $a$. In this paper we will find $\D(G)$ where $G$ is a simple Zassenhaus group.
\end{abstract}
\openup 1.2\jot



\maketitle
\section{Introduction and Prelimits}

There are several graphs associated to algebraic structures, specially finite groups, and many interesting results have been obtained recently (see for example~\cite{Da, HI, IP, KG} ). In~\cite{CC} a new graph namely \emph{divisibility graph} which is related to a set of positive integers have been introduced. The divisibility graph, $\overrightarrow{D}(X)$ is a graph with vertex set $X^*=X \setminus \{1\}$ and there is an arc between two vertices $a$ and $b$ if and only if $a$ divides $b$. Since this graph is never strongly connected, we consider its underlying graph $\D(X)$ without changing the name for convenience. Also we use $\D(G)$ instead of  $\D(\cs(G))$ where $\cs(G)$ is the set of conjugacy class sizes of elements of $G$.

All groups considered here are finite. Let $H$ be a subgroup of a group $G.$ We use $[G:H]=\{Hg|~g\in G\}$
to denote the right cosets of $H$ in $G.$ The conjugacy class of an element $g\in G$, will
denote by $g^G$ which is the set of all $x^{-1}gx$ where $x\in G.$ Also we use $\cs(G)$ for the set of conjugacy class sizes of elements of $G.$
The conjugate of $h\in H$ by $g\in G$ and the conjugate of $H$ by $g$ are $h^g=g^{-1}hg$ and
$H^g=g^{-1}Hg$ respectively. Also by $C_G(g)=\{x\in G|~g^x=g~\}$, $Z(H)=\{z\in G|~g^z=g~for~every~g\in G~\}$ and $N_G(H)=\{g\in G|~H^g=g~\}$, we denote the centralizer of $g$ in $G$,
the center of $G$ and the normalizer of $H$ in $G$ respectively. An involution of $G$ is an element of order $2$. By $H$ is disjoint from its conjugates we mean that $H\cap H^g=\{e\}~or~H$, for every $g\in G$. We write $G=<g>$, to show that $G$ is generated by $g$. For a finite set $X$ we use $|X|$ for cardinality of $X$ and $\gcd(a,b)$ is the greatest common divisor of $a$ and $b$ where $a$ and $b$ are positive integers. For undefined terminology and notation about groups we refer to \cite{R}.

Throughout the paper all graphs are finite and simple. Let $\Gamma_1, \Gamma_2$ be two graphs then $\Gamma= \Gamma_1+ \Gamma_2$ means $\Gamma$ is disjoint union of
$\Gamma_1$ and $\Gamma_2$. The complete graph with $n$ vertices will denote by $\K_n$. For other notation about graphs we refer to \cite{we}.

The divisibility graph $\D(G)$, where $G$ is a symmetric group or an alternating group studied in~\cite{AI}. It is shown there that $\D(S_n)$ has at most two connected components
and if it is not connected then one of its connected components is $\K_1$. Also we have proved that $\D(A_n)$ has at most three connected components and if it is not connected then
two of its connected components are $\K_1$.

In this paper we consider $\D(G)$ where $G$ is a simple Zassenhaus group.

\begin{definition}\quad\cite[Chapter~13]{Go}\label{z1}
A permutation group $G$ is a Zassenhaus group if it is  doubly transitive
 in which only the identity fixes three letters.
\end{definition}
\begin{theorem}\quad\cite[Chapter~13]{Go}\label{z2}
Let $G$ be a Zassenhaus group of degree $n+1$ and let $N$ be the subgroup
fixing a letter. Then we have
\begin{enumerate}
\item[(i)] $N$ is a Frobenius group with kernel $K$ of order $n$ and complement $H$.
\item[(ii)]  $K$ is a Hall subgroup of $G$, $K$ is disjoint from its conjugates, $C_G(k_0)\leq K$ for every $k_0\in K$ and
$N = N_G(K)$.
\item[(iii)] $H$ is a subgroup of $G$ fixing two letters, $H$ is disjoint from its
conjugates in $G$, and $[N_G(H) :H] = 2$.
\item[(iv)] $|G| = en(n + 1)$, where $e = |H|$ and $e$ divides $n - 1$.
\item[(v)] If $e$ is even then $K$ is abelian.
\item[(vi)] If $e\geq(n - 1)/{2}$ then $K$ is an elementary abelian $p$-group for some
prime $p$.
\item[(vii)] If $G$ is simple and $e$ is odd then $H$ is cyclic. Also $G$ has only one conjugacy class of involutions of size $e(n + 1)$ if $n$ is even and of size $en$ if
$n$ is odd.
\end{enumerate}
In this case we say that $G$ is of type $(H,K)$.
\end{theorem}
Also $G$ has only one conjugacy class of involutions of size $e(n + 1)$ if $n$ is even and of size $en$ if
$n$ is odd.

In this case we say that $G$ is of type $(H,K)$.
As it is stated in \cite[Chapter~16, Page 488]{Go}, $\PSL(2, q)$, the projective special linear group of dimension two over a finite field $\mathbb{F}_q$ of order $q>3$, and
$\Sz(q)$, Suzuki simple groups of order $q^2(q-1)(q^2+1)$ where $q=2^m$ and $m>1$ is an odd number are the only simple
Zassenhaus groups. (For more details about the structure of $\PSL(2, q)$ and $\Sz(q)$ we refer to~\cite{D} and \cite{S} respectively.)

According to the notations of Theorem~\ref{z2}, the projective special linear groups $\PSL(2,q)$ are Zassenhaus group of type $(H,K)$ of degree $n+1$ where
$|K|=n=q$, $e=(n-1)/{2}$ if $n$ is odd, and $e=n-1$ otherwise.
Also  $H$ is cyclic group that inverted by an involution \cite[Chapter~13]{Go}.
So by Theorem \ref{z1}, $|G|=q(q^2-1)/2$ if $q$ is odd and $|G|=q(q^2-1)$ otherwise. It is proved in~\cite[Chapter~XII.]{D} that the projective special linear groups, $\PSL(2,q)$, has a cyclic subgroup namely $L$ of
order ${(q+1)}/{\gcd(2,q+1)}$ such that $x^{\PSL(2,q)}\cap L=\{x,x^{-1}\}$ for every $x\in L$. $L$ is disjoint from its conjugates in $G$ and $[N_G(L):L]=2$.

Also by Theorem~\ref{z2}, we see that the Suzuki simple groups, $\Sz(q)$, are Zassenhaus group of type $(H, K)$ with $H$ is cyclic of order $q - 1$ and $K$ is a
nonabelian 2-group of order $q^2$ with elementary abelian center of order $q$ and $N=HK$ is the subgroup fixing a letter~\cite[Page ~466]{Go}. So by Theorem~\ref{z2},
$|G|=q^2(q-1)(q^2+1)$ where $q=2^m$ and $m>1$ is an odd number.

Not that by~\cite{FP} the Suzuki simple groups, $\Sz(q)$, can be defined as a subgroup of the group $\SL_4(q)$. In this representation,
$K=<(\alpha,\beta)~|~\alpha,~\beta\in \mathbb{F}_q>$ and $Z(K)=\{(0,\beta)~|~\beta\in \mathbb{F}_q\}$ where $(\alpha,\beta)$ is defined
\[
(\alpha,\beta)=
\left[ {\begin{array}{cccc}
1 & 0 & 0 & 0 \\
\alpha & 1 & 0 & 0 \\
\alpha^{1+r}+\beta & \alpha^r & 1 & 0 \\
\alpha^{2+r}+\alpha\beta+\beta^r & \beta & \alpha & 1
\end{array} } \right],
\]
with multiplicity $(\alpha,\beta)(\gamma,\delta)=(\alpha+\gamma,\alpha\gamma^r+\beta+\delta)$ and $r^2=2q$.
Note that this implies that all $1\neq x\in Z(K)$ are involutions.

The following three lemmas about Suzuki groups will be useful.
\begin{lemma}\quad\cite[Proposition~16]{S}\label{s1}
$\Sz(q)$ contains abelian subgroups $A_0,~A_1$ and $A_2$ of order $q-1$ and $q+r+1$ and $q-r+1,~(r^2=2q)$, respectively. $C_G(x)=A_i$ for every $x\in A_i$ and $i=0,1,2$. Moreover  $[N_G(A_i):A_i]=4$ for $i=1,2$.
\end{lemma}
\begin{lemma}\quad\cite[Lemma~3.3]{FP}\label{s4}
The subgroups $A_i$'s are disjoint from their conjugates for $i=0,~1,~2$.
\end{lemma}
By Theorem~\ref{z2}, we see that $H=A_0$.
\begin{lemma}\quad\cite[Lemma~3.2., Part~(h)]{S}\label{s2}
Let $x=(\alpha,\beta)\in K$, where $\alpha\neq 0$. Then $C_G(x)=<(\alpha,0),Z(K)>$, where $Z(K)=\{(0,\xi):\xi \in \mathbb{F}_q\}$.
\end{lemma}
Now we prove the following lemma.
\begin{lemma}\quad\label{z3}
Let $G$ be a group and $H$ a subgroup of $G$ which is disjoint from its conjugates in $G$. Suppose $h$ is a non identity element of $H$,  $N_G(H)=N$, and $C_G(h)=C$.
Then $|N|=|h^G\cap H||C|$.
\end{lemma}
\begin{proof}
Let $c\in C$. So $h^c=h$ and $H^c\cap H\neq \{e\}$. Since $H$ is disjoint from its conjugates, we can conclude $H^c=H$. Hence $c\in N$ and $C\leq N$. Define a map
$\theta$ from $[N:C]$ to $h^G\cap H$, by $\theta(Cg)=h^g$. Suppose $g_0$ is an arbitrary element of $G$. Since $H$ is disjoint from its conjugates, we conclude $h^{g_0}\in H$
if and only if $g_0\in N$. Also $Cx=Cy$ for $x, y\in N$ if and only if $xy^{-1}\in C$, which implies that $\theta$ is well defined and injective. It is easy to see that
$\theta$ is also surjective. This completes the proof.
\end{proof}

In the next section we will calculate the conjugacy class sizes of elements of the projective special linear groups $\PSL(2,q)$ and the Suzuki groups
$\Sz(q)$, when they are simple. The structure of divisibility graph for these groups will also investigate.

\section{The structure of divisibility graph for simple Zassenhaus groups}
\label{sec:main}
This section includes two theorems. We will determine the structure of divisibility graphs for $\PSL(2,q)$ and $\Sz(q)$ when they are simple.
\begin{theorem}\quad\label{p2}
Let $G=\PSL(2,q)$. Then $\D(G)$ is one of the graphs of the following list:
\begin{enumerate}
\item[(i)] $3\K_1$
\item[(ii)] $\K_2+2\K_1$.
\end{enumerate}
\end{theorem}
\begin{proof}
To calculate the conjugacy class sizes of non identity elements of $G$, we consider three subgroups $H,~K$ and $L$ which their orders are pairwise coprime.
So the non identity elements of these groups are not conjugate.

First consider the subgroup $H$.
Let $H=<h_0>$. By the structure of $G$, we know that $H$ is inverted by an involution.
So every elements of $H$ is cinjugate to its inverse and $N_G(H)\neq C_G(H)$. Also $H\leq C_G(h_0)=C_G(H)$.
Since by part (iii) of Theorem~\ref{z2}, $[N_G(H):H]=2$, then $C_G(h_0)=H$.
This implies $|h_0^G|=q(q+1)$.

Consider an arbitrary element $1\neq h\in H$. Clearly $C_G(h_0)\leq C_G(h)$.
If $C_G(h_0)= C_G(h)$ then $|h^G|=q(q+1)$ and $h$ is different from its inverse as $h_0$ is.
If $C_G(h_0)\neq C_G(h)$
then, since $H$ is disjoint from its conjugate, $C_G(h)=N_G(H)$.
So by Lemma~\ref{z3}, $h$ has only one conjugate in $H$.
Hence  $h=h^{-1}.$ This means that $h$ is an involution and $|h^G|={q(q+1)}/{2}$.
Therefore if $e$ is even, there exist one conjugacy classes of size ${q(q+1)}/{2}$ related to $h_0^{{e}/{2}}$
and ${(e-2)}/{2}$ conjugacy class of size $q(q+1)$ in $G$, related to other non identity elements of $H$.
If $e$ is odd then $G$ has $(e-1)/{2}$ conjugacy classes of size $q(q+1)$ that contains non identity elements of $H$.

Now consider the subgroup $K$. First suppose $q=2^\alpha$.
In this case by Theorem \ref{z2} parts (vi) and (vii), $K$ is an elementary abelian 2-group and
$G$ has only one conjugacy class of involutions.
So every element of $K$ contains in the unique conjugacy class of involutions of size $(q^2-1)$.

Suppose $q=p^\alpha$ is an odd number. By parts (ii) and (vi) of Theorem \ref{z2}, $K=C_G(y)$ for each $y\in K$.
Since $[N:K]=e$, by Lemma~\ref{z3}, $|y^G\cap K|=e$.
So we can partition $K\setminus\{1\}$ into two disjoint subset of length $e$, each of them contains conjugate elements.
Hence in this case, there are two conjugacy classes in $G$ of size $(q^2-1)/2$ that contains elements of $K$.

Finally consider the subgroup $L$. Since for every $x\in L$ we have $x^{\PSL(2,q)}\cap L=\{x,x^{-1}\}$ by Lemma~\ref{z3},
if $x\in L$ is not an involution then we have $|C_G(x)|=|N_G(L)|/2=|L|$ and if $x\in L$ is an involution then  $|C_G(x)|=|N_G(L)|=2|L|$.
So $|L|$ is even then there exist $(|L|-2)/2$ conjugacy classes in $G$ of size
$\gcd(2,q+1)eq$ and one conjugacy class of size $\gcd(2,q+1)eq/2$ that contain elements of $L$.
If $|L|$ is odd, then there exist $(|L|-1)/2$ conjugacy classes in $G$ of size $\gcd(2,q+1)eq$ that contain elements of $L$.

We must consider two following cases:\\
(Case 1) $G=\PSL(2,2^k)$. In this case $e=n-1=q-1$ is an odd number. By using class equation of $G$  we have;
\setlength\arraycolsep{1.4pt}
\begin{eqnarray*}
1\ +\ {(|L|-1)}/{2}\ .\ eq\ +\ (q-1)(q+1)\ +\ (({e-1})/{2})\ .\ q(q+1)\ &=& \\1\ +\ {q^2(q-1)}/{2}\ +\ (q^2-1)\ +\ {(q-2)q(q+1)}/{2}\ &=& \ |G|.
\end{eqnarray*}
(Case 2) $G=\PSL(2,p^k)$ where p is an odd prime. In this case $e=(n-1)/{2}=(q-1)/{2}$. First assume that $e$ is odd. In this case we have;
\setlength\arraycolsep{1.4pt}
\begin{eqnarray*}
1\ +\ ((|L|-2)~/~2)~.\ 2eq\ +\ eq\ +\ 2\ .\ e(q+1)\ +\ ({(e-1)}/{2})\ .\ q(q+1)\ &= & \\ 1\ + {q(q-1)(q-3)}/{4}\ +\ {q(q-1)}/{2}\ +\ (q^2-1)\ +\ {(q-3)q(q+1)}/{4}\ &=& |G|.
\end{eqnarray*}
In case where $e$ is even we have;
\setlength\arraycolsep{1.4pt}
\begin{eqnarray*}
1\ +\ ({(|L|-1)}/{2})\ .\ 2eq\ +\ 2\ . e(q+1)\ +\ ({(e-2)}/{2})\ .\ q(q+1)\ +\ {q(q+1)}/{2}&=& \\ \ \ \ \ 1\ +\ {q(q-1)^2}/{4}\ +\ (q^2-1)\ +\ {(q-5)q(q+1)}/{4}\ +\ {q(q+1)}/{2}&=&|G|.
\end{eqnarray*}
Hence by calculating all conjugacy classes of $\PSL(2,q)$, we see that $\D(\PSL(2,q))$ is $3\K_1$ or $\K_2+2\K_1$. (See Figure \ref{f1}.)
\end{proof}
\begin{figure}[here]
\begin{center}
\includegraphics[height=3 cm]{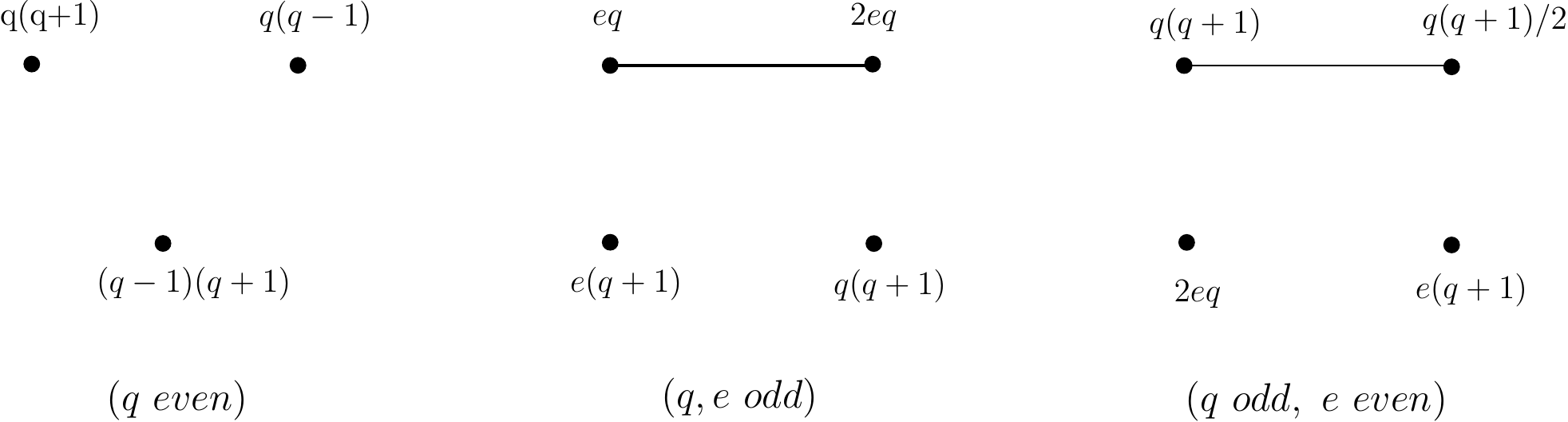}
\caption{The graph $\D(\PSL(2,q))$}
\label{f1}
\end{center}
\end{figure}
Now we calculate the conjugacy class sizes of Suzuki simple groups.
\begin{theorem}\quad\label{s3}
Let $G=\Sz(q)$. Then $\D(G)=\K_2+3\K_1$.
\end{theorem}
\begin{proof}
By Lemmas~\ref{s1} and \ref{s2}, $G$ has four subgroups $H,~K,~A_1$ and $A_2$ with pairwise coprime orders.

We start by calculating the conjugacy class sizes of elements of $H$.
By Lemma \ref{s1}, $C_G(h)=H$ for every $1\neq h\in H$.
Since $H$ is disjoint from its conjugates, by Lemma \ref{z3}, $h$ has $[N_G(H):C_G(h)]=2$ conjugates in $H$.
So there are ${(q-2)}/{2}$ conjugacy classes in $G$ of size $q^2(q^2+1)$ which contains elements of $H$.

Consider the conjugacy classes of elements of $K$. Let $x_0=(\alpha,\beta)\in K$ be an arbitrary element.
If $1\neq x_0\in Z(K)$ then $x_0$ is an involution and by Theorem \ref{z2} part (vii), $G$ has only one class of involutions of size $(q-1)(q^2+1)$.
If $x_0\in K\setminus Z(K)$ then by Lemma~\ref{s2}, $C_G(x_0)=\{(x,y)|~y\in \mathbb{F}_q,~x=0~or~\alpha\}$.
This means that $|C_G(x_0)|=2q$ and $|x_0^G|={q}(q-1)(q^2+1)/2$.
Since $K$ is disjoint from its conjugates, by Lemma \ref{z3}, $x_0$ has $[K:C_G(x_0)]={q}(q-1)/{2}$ conjugates
in $K\setminus Z(K)$. So there exist ${(q^2-q)}/{({q}(q-1)/{2})}=2$ conjugacy classes of
size ${q}(q-1)(q^2+1)/{2}$ in $G$ which contains elements of $K\setminus Z(K)$.

Finally we consider two subgroups $A_1$ and $A_2$.
By Lemma \ref{s1}, $A_1$ and $A_2$ are of order $n_1=q+r+1$ and $n_2=q-r+1$ respectively.
Also $[N_G(A_i):A_i]=4$ and for every $x\in A_i$, $C_G(x)=A_i$ where $i=1,2$.
So for every $x\in A_i,~i=1,2,$ we have $[N_G(A_i):C_G(x)]=4$ and by Lemmas \ref{s4} and \ref{z3},
there are just four elements of $A_i$ which are conjugate with $x$ in $G$.
Hence there are ${(q+r)}/{4}$ conjugacy classes of size ${q^2(q-1)(q^2+1)}/{(q+r+1)}=q^2(q-1)(q-r+1)$ which
contains elements of $A_1$ and ${(q-r)}/{4}$ conjugacy classes of size
${q^2(q-1)(q^2+1)}/{(q-r+1)}=q^2(q-1)(q+r+1)$ which contains elements of $A_2$.

By using the class equation for $G$ we see that
\begin{eqnarray*}
&&1~+~{(q-2)}~.~q^2(q^2+1)/{2}~+~(q-1)(q^2+1)~+~2~.~{q}(q-1)(q^2+1)/{2}+
\\&& {(q+r)}~.~q^2(q-1)(q-r+1)/{4}~+~{(q-r)}~.~q^2(q-1)(q+r+1)/{4}=|G|.
\end{eqnarray*}
This implies that $\D(\Sz(q))=\K_2+3\K_1$. (See Figure \ref{f2}.)
\end{proof}
\begin{figure}[here]\label{f2}
\begin{center}
\includegraphics[height=3 cm]{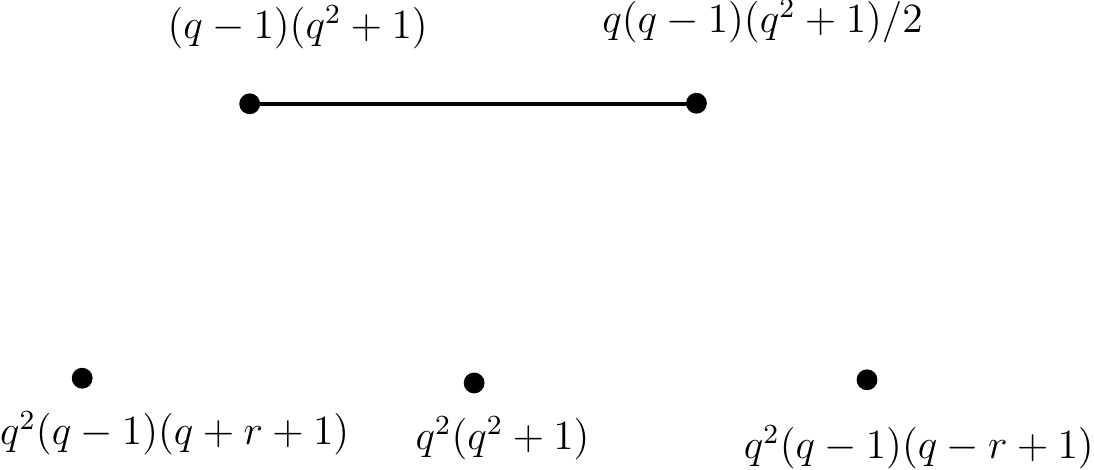}
\caption{The graph $\D(\Sz(q))$}
\end{center}
\end{figure}

\end{document}